\documentclass[11pt]{article}
\usepackage{amsfonts}
\usepackage{amsmath,amssymb}

\usepackage{pstricks,pstricks-add}

\usepackage{anysize} 
\marginsize{2cm}{2cm}{2cm}{2cm} 

\allowdisplaybreaks
\newtheorem{theorem}{Theorem}

\newtheorem{corollary}[theorem]{Corollary}

\newenvironment{proof}[1][Proof]{\noindent\textbf{#1.} }{\ \rule{0.5em}{0.5em}}

\begin{document}

\title{\bf Certain fractional Laplacian equations\\ that do not have smooth solutions}
\author{{\bf Jos\'e Villa-Morales}\\
Universidad Aut\'{o}noma de Aguascalientes\\
Departamento de Matem\'{a}ticas y F\'{\i}sica\\
Av. Universidad No. 940, Cd. Universitaria \\
Aguascalientes, Ags., C.P. 20131, M\'exico\\
\ttfamily{jvilla@correo.uaa.mx}}
\date{}
\maketitle

\begin{abstract}
Let $f$ be a real-valued function defined on $\mathbb{R}$, with $f(0) \neq 0$ and which is not constant in non empty open intervals. We prove the equations
\begin{equation}\label{edif} 
\left\{ 
\begin{array}{rcll}
(-\Delta )^{s}u & = & f(u), & \text{in }B_{1}, \\ 
u & = & 0, & \text{in }B_{1}^{c},
\end{array}
\right.
\end{equation}
where $(-\Delta )^{s}$ is the $s$-fractional Laplacian,  $0< s <1$, have no solutions in $C^{2}(\overline{B_{1}})$, if $d>2s$. The proof is based on the moving plane method and in the approximation of $C^{2}(\overline{B_{1}})$ functions by $s$-harmonic functions.\\

\end{abstract}

\section{Introduction}

By $B_{r}(a)\subset \mathbb{R}^{d}$ we are going to denote the open ball with center at $a$ and radious $r>0$. When $a=0$ we simply write $B_{r}$ instead of $B_{r}(0)$. 

A natural framework space for the solutions of equations like (\ref{edif}) could be the domain of the operator $(-\Delta )^{s}$. However, we do not have a characterization of such space, and this is a disadvantage if we want to study properties of the solutions of these equations.

If the reaction term $f$ in (\ref{edif}) is smooth then we are able to obtain a regular solution. For example, the H\"{o}lder function $u(x)=c(1-|x|^{2})_{+}^{s}$, $c>0$, is the unique solution of Dirichlet problem
\begin{equation}
\left\{ 
\begin{array}{rcll}
(-\Delta )^{s}u & = & 1, & \text{in }B_{1}, \\ 
u & = & 0, & \text{in }B_{1}^{c},
\end{array}
\right.
\end{equation}
see for example \cite{Get}.  This example shows us that we can obtain solutions of equations like (\ref{edif})  in specific function spaces. One more example, using variational techniques is proved in \cite{SeVa1} that the equation
\begin{equation}\label {BreNe}
\left\{ 
\begin{array}{rcll}
(-\Delta )^{s}u & = & \lambda u+|u|^{2^{*}-2}u, & \text{in }\Omega, \\ 
u & = & 0, & \text{in }\Omega^{c},
\end{array}
\right.
\end{equation}
has a solution in certain fractional Sobolev space, $H^{s}(\mathbb{R}^{d})$, where $2^{*}=2d/(d-2)$, $d>2$, and $\Omega \subset \mathbb{R}^{d}$ is an open bounded set. 

Moreover,  in \cite{Fall, SeVa2} is studied that certain equations, that are generalization of (\ref{BreNe}), have solutions in a weak sense (distributional sense). On the other hand, in \cite{RosSerra} is proved a Pohozaev identity for the solution of an equation like (\ref{edif}), for some reaction functions $f$, and it is used to show that such equation does not have weak solutions. An analogous result is obtained in \cite{FallWeth} using the method of moving spheres and the Caffarelli-Silvestre extension technique.

Besides, in order to study properties of solutions of equations like (\ref{edif}) it is usually assumed that such solutions exist and that they are regular in some sense, this is the case for example in the study of radial symmetric solutions, see \cite{CLL}. Moreover, it is well known that $s$-harmonic functions are $C^{\infty}$ in the domain where they are fractional harmonic \cite{CSS}, so that imposing certain conditions of smoothness at the reaction term one could expected that these solutions are also regular \cite{CFY, BLMW}. 

The main contribution of the present work is to show that if the reaction term, $f$ in (\ref{edif}), belongs to the space of real-valued functions
$$\mathcal{J}=\left\{ f:\mathbb{R} \rightarrow \mathbb{R} \  \vert  f \text{ is not constant in non empty open intervals}\right\}$$
and $f(0)\neq 0$ then equations (\ref{edif}) have no solutions on $C^{2}(\overline{B_{1}})$. Therefore, due to the previous discussion, we can conclude that solutions of equations of type (\ref{edif}) could be continuous, of H\"older type or solutions in a weak sense, but not smooth.

Usually, to demonstrate properties of the solutions of equations of type (\ref{edif}) the method of moving planes is used. This is the case, for example, when is proved that they are radially symmetric. As is known, this technique is based on the Maximum Principle and Harnack inequality, see \cite{PR} for a more complete discussion of the method and its various applications. However, in our case we will use the method of moving planes, but surprisingly we will not use the Maximum Principle, instead we will use the recent result of approximation of smooth functions by $s$-harmonic functions \cite{DSV,Kry}. This approximation theorem has proved to be quite useful and with it is achieved, transparently, to discard the space of smooth functions as a possible space where we can find solutions of equations of type (\ref{edif}).

The importance of the study of fractional equations is well known in applied mathematics. For example, they arise in fields like molecular biology \cite{SZF}, combustion theory \cite{CRS}, dislocations in mechanical systems \cite{IM}, crystals \cite{Toland} and in models of anomalous growth of certain fractal interfaces \cite{MW}, to name a few.

The paper is organized as follows. In Section 2 we address the approximation theorem of smooth functions by $s$-harmonic functions and in Section 3 we enunciate and demonstrate the main result of the paper.

\section{Preliminaries}

Let us introduce some notation. Let $U\subset \mathbb{R}^{d}$ be an open set, and $k\in \{0,1,2,...\}$,
\begin{eqnarray*}
C^{k}(U)&=&\{\varphi:U\rightarrow \mathbb{R} \ \vert \ D^{\alpha}\varphi \text{ is continuous on } U, \text{ for all } |\alpha| \leq k\},\\
C^{k}(\overline{U})&=&\{\varphi\in C^{k}(U) \ \vert \ D^{\alpha}\varphi \text{ is uniformly continuous on bounded subsets of } U,\\ &&\text{for all } |\alpha| \leq k\},
\end{eqnarray*}
and if $\varphi\in C^{k}(\overline{U})$ is bounded we write 
$$||\varphi||_{C^{k}(\overline{U})}=\sum_{|\alpha|\leq k} ||D^{\alpha}\varphi||_{C(\overline{U})},$$
where $C(\overline{U})=C^{0}(\overline{U})$ and  $||\varphi||_{C(\overline{U})}=\sup \{|\varphi(x)|:x\in U\}$.

In this section we are going to prove that each function on $u\in C^{k}(\overline{B_{r}(a)})$ can be approximated by $s$-harmonic functions on $B_{r}(a)$. This result was first proved in \cite{DSV}. Recently, Krylov \cite{Kry} has shown this result using an integral representation for the $s$-harmonic functions. For completeness we give a proof of the approximation theorem following the cleaver ideas given in \cite{Kry} considering some minor changes, such as the use of a smooth version of Stone-Weierstrass theorem.

The fractional Laplacian $(-\Delta)^{s}$, $0<s<1$, of a function $\varphi:\mathbb{R}^{d}\rightarrow \mathbb{R}$ is 
\begin{equation}
\mathcal{F}((-\Delta )^{s}\varphi) (\xi)= (2\pi |\xi |)^{2s}\mathcal{F}\varphi(\xi), \ \ \xi \in \mathbb{R}^{d}, \label{defLa}
\end{equation}
where $\mathcal{F}$ is the Fourier transform, 
$$\mathcal{F}\varphi(\xi)=\int \varphi(x) e^{-2\pi i \langle x,  \xi \rangle}dx,\ \ \ \xi\in \mathbb{R}^{d},$$
where $\langle \cdot,\cdot\cdot \rangle$ is the usual inner product in $\mathbb{R}^{d}$. Let $\mathcal{S}(\mathbb{R}^{d})$ be the Schwartz space of rapidly decaying functions. If $\varphi\in \mathcal{S}(\mathbb{R}^{d})$ then 
\begin{equation*}
(-\Delta )^{s}\varphi (x) =c_{d,s} \ \text{P.V.} \ \int_{\mathbb{R}^{d}}\frac{\varphi (x)-\varphi (y)}{|x-y|^{d+2s}}dy,
\end{equation*}
where  $c_{d,s}>0$ is a normalization constant and P.V. is the Cauchy principal value. 

If $K_{-\alpha}(x)=|x|^{-\alpha}$, with $0 <\alpha <d$, then
\begin{equation}
\mathcal{F}(K_{-\alpha})(\xi )=\frac{(2\pi )^{\alpha}}{c_{d,\alpha}}|\xi |^{\alpha-d}. \label{tfonor}
\end{equation}
In what follows we will assume that $d>2s$. Using (\ref{tfonor}) and (\ref{defLa}) we get
\begin{eqnarray}
\varphi (x) &=&\mathcal{F}^{-1}((2\pi |\xi |)^{-2s}\mathcal{F}((-\Delta )^{s}\varphi)(\xi))(x) \nonumber \\
&=&c_{d,s}\int_{\mathbb{R}^{d}}|y-x|^{2s-d}(-\Delta )^{s}\varphi(y)dy, \label{repintsol}
\end{eqnarray}
where $\mathcal{F}^{-1}$ is the inverse Fourier transform. For a function $g:\mathbb{R}^{d}\rightarrow \mathbb{R}$ let us introduce the function
\begin{equation*}
K_{2s-d}g(x)=(K_{2s-d}\ast g)(x)=\int_{\mathbb{R}^{d}}K_{2s-d}(x-y)g(y)dy, \ \ x\in \mathbb{R}^{d}.
\end{equation*}
Using the representation (\ref{repintsol}), of $\varphi$, as a motivation is introduced in \cite{Kry} the following space of real-valued functions
\begin{equation*}
\mathcal{K}=\{f:B_{1}\rightarrow \mathbb{R} \ \vert \text{ there exists }g\in C_{0}^{\infty
}(U),\ \text{such that }f=K_{2s-d}g\}, 
\end{equation*}
where $U=B_{4}\backslash \overline{B}_{3}$.

\begin{theorem}\label{TeoApx}
If $d>2s$, then $\mathcal{K}$ is a linear subspace dense in $C^{k}(\overline{B_{1}})$, for each $k\geq 0$.
\end{theorem}
\begin{proof}
From the properties of convolution $*$ and since the support, spt$(g_{1}+g_{2})$, of $g_{2}+g_{2}$ is contained in spt$(g_{1}) \ \cup$ spt$(g_{2})$ we see that $\mathcal{K}$ is a linear subspace. 

For each $x\in U$, fixed, let us see that $K_{2s-d}^{x}\in \overline{\mathcal{K}}$, where $K_{2s-d}^{x}(y)=K_{2s-d}(x-y)$. There exits $r=r(x)>0$ such that $\overline{B_{r}(x)}\subset U$. Therefore we can take $\zeta \geq 0$ in $C_{0}^{\infty }(B_{1})$ which integrates one and $\xi \in C_{0}^{\infty }(U)$ such that $0\leq \xi \leq 1$, $\xi =1$ on $\overline{B_{r}(x)}$, and spt$(\xi )\subset U$. Let us consider the sequence $(f_{m})$, where $f_{m}=K_{2s-d}g_{m}$ and
\begin{equation*}
g_{m}(z)=\xi (z)\zeta _{m}(z-x),\ \ z\in \mathbb{R}^{d}, 
\end{equation*}%
here $\zeta _{m}(z)=m^{d}\zeta (mz)$. By construction $g_{m}\in C_{0}^{\infty }(U)$ and the sequence $(f_{m})$ converges in $C^{k}(\overline{B_{1}})$ to $K_{2s-d}^{x}$. 

For each $y\in B_{1}$ and $x\in U$ we see, for $i=,...,d$,
\begin{equation}
(2s-d)(x_{i}-y_{i})K_{2s-d-2}^{x}(y)=\frac{\partial K_{2s-d}^{x}(y)}{\partial x_{i}} =\lim_{m\rightarrow \infty }m\left[K_{2s-d}^{x+\tfrac{1}{m}e_{i}}(y)-K_{2s-d}^{x}(y)\right].  \label{dep}
\end{equation}
For $m$ big enough $x+\tfrac{1}{m}e_{i}\in U$ and the limit in $C^{k}(\overline{B_{1}})$ implies $\frac{\partial
K_{2s-d}^{x}(\cdot)}{\partial x_{i}}\in \overline{\mathcal{K}}$. In particular, we get that $\Delta _{x}^{m}K_{2s-d}^{x}(\cdot)=\Delta _{x}(\Delta_{x}^{m-1}K_{2s-d}^{x}(\cdot))\in \overline{\mathcal{K}}$, for each $m\in \mathbb{N}$, where $\Delta _{x}$\ is the Laplacian applied with respect to $x$. Inasmuch as 
\begin{equation*}
\Delta_{x}^{m}K_{2s-d}^{x}(y)=K_{(2s-d)-2m}^{x}(y)\prod\limits_{j=0}^{m}2(s-1-j)(2s-d-2j), 
\end{equation*}
and $d>2s$ implies $K_{(2s-d)-2m}^{x}(\cdot)\in \overline{\mathcal{K}}$ for each $m=0,1,2,....$

Taken $x\in U$ and $t>0$ the Weierstrass $M$-test implies (see Lemma 2.2 in \cite{Kry}) 
\begin{equation*}
K_{2s-d}^{x}(\cdot )\exp \left(-t|x-\cdot |^{-2}\right) =\sum_{m=0}^{\infty }\frac{1}{m!}(-t)^{m}K_{(2s-d)-2m}^{x}(\cdot )\in \overline{\mathcal{K}}.
\end{equation*}
For any $x\in U$ and $\alpha \in (-1,\infty )$ the integral
\begin{equation*}
K_{2s-d+2\alpha +2}^{x}(\cdot )\Gamma (\alpha +1) =\int_{0}^{\infty }t^{\alpha }K_{2s-d}^{x}(\cdot )\exp \left( -t|x-\cdot|^{-2}\right) dt
\end{equation*}
converges in $C^{k}(\overline{B_{1}})$, then $K_{2s-d+2\alpha +2}^{x}(\cdot )\in \overline{\mathcal{K}}$. Taking ($d>2s$) 
\begin{equation*}
\alpha  \in  \left\{ \frac{d-2-2s}{2}, \frac{d-2s}{2}, \frac{2+d-2s}{2}\right\}
\end{equation*}
we conclude that the constant functions, $K_{2}^{x}(\cdot)$ and $K_{4}^{x}(\cdot )$ are in $\overline{\mathcal{K}}$. From (\ref{dep}) we can deduce that, for each $i,j\in \{1,...,d\}$,
\begin{eqnarray*}
y_{i} &=&x_{i}-\frac{1}{2}\frac{\partial K_{2}^{x}(y)}{\partial x_{i}}, \\
y_{i}y_{j} &=&-x_{i}x_{j}+x_{j}y_{i}+x_{i}y_{j}+\frac{1}{8}\frac{\partial^{2}K_{4}^{x}(y)}{\partial x_{j}\partial x_{i}},
\end{eqnarray*}%
are in $\overline{\mathcal{K}}$. Then $\overline{\mathcal{K}}$ contains the polynomials in the $d$ coordinates. The result follows from the Stone-Weierstrass theorem for smooth functions (see, for example, Corollary 6.3 and Proposition 7.1 in the Appendixes of \cite{EK}). \hfill
\end{proof}

\begin{corollary}\label{CorApx}
Let us assume $d>2s$ and $k\geq 0$. If $u\in C^{k}(\overline{B_{r}(a)})$, then for each $\varepsilon>0$ there exists  $h\in C^{\infty}_{0}(B_{4r}(a) \backslash \overline{B_{3r}(a)})$ such that  $||u-K_{2s-d}h||_{C^{k}(\overline{B_{r}(a)})}<\varepsilon$.
\end{corollary}
\begin{proof}
Applying Theorem \ref{TeoApx} to the function $v:B_{1}\rightarrow \mathbb{R}$, $v(x)=r^{-2s}u(rx+a)$, we have that there exists $g\in C^{\infty}_{0}(B_{4} \backslash \overline{B_{3}})$ such that $||v-K_{2s-d}g||_{C^{k}(\overline{B_{1}})}< r^{1-2s}\varepsilon$. Let us define $h:\mathbb{R}^{d}\rightarrow \mathbb{R}$ as $h(z)=g\left(\frac{1}{r}(z-a)\right)$, then $h \in C^{\infty}_{0}(B_{4r}(a) \backslash \overline{B_{3r}(a)})$. Since 
$$K_{2s-d}h(z)=r^{2s}(K_{2s-d}g)\left(\frac{1}{r}(z-a)\right),\ \ z\in \mathbb{R}^{d},$$ then, for each $|\alpha|\leq k$,
$$D^{\alpha}(u-K_{2s-d}h)(z)=r^{2s-1}D^{\alpha}(v-K_{2s-d}g)\left(\frac{1}{r}(z-a)\right).$$
From which the result is followed.\hfill
\end{proof}

\section{The main result}

Here we will present the proof of our main contribution. First let us discard a trivial case. If the reaction term satisfies $f(0)=0$, then $u\equiv 0$ is a solution of (\ref{edif}). 

\begin{theorem}
If $f(0)\neq 0$ and $f\in \mathcal{J}$, the equation (\ref{edif}) has no solution in $C^{2}(\overline{B_{1}})$.
\end{theorem}

\begin{proof}
Let $a\in \mathbb{R}^{d}\backslash \{0\}$ be an arbitrary fixed point. We are going to consider the affine hyperplane 
\begin{equation}
H_{a}=\{z\in \mathbb{R}^{d}:\langle z,a\rangle =\langle a,a\rangle \} \label{defref}
\end{equation}
that goes through $a$ and is also perpendicular to $a$. The reflection with respect to $H_{a}$ of a point $x\in \mathbb{R}^{d}$ is defined as
\begin{equation*}
x_{r}=2a+x-\frac{2\langle a,x\rangle }{\langle a,a\rangle }a,
\end{equation*}
see the Figure 1.
\begin{center}
\begin{pspicture}(0,0)(8,8)
 \psline{-}(2,7)(7,2)
 \psline{->}(2.5,2.5)(4.5,4.5)
 \rput(1.5,4.2){$y$}
 \psset{dotsize=2.5pt 0}
 \psdots(2.5,2.5)
 \psline{->}(2.5,2.5)(2,4.5)
 \psline{-,linestyle=dashed}(2,4.5)(4.5,7)
 \psline{-,linecolor=blue}(4.5,7)(5,2.5)
 \psline{-,linestyle=dashed}(5,2.5)(6.5,4)
 \psline{-,linecolor=blue}(6.5,4)(2,4.5)
 \psline{->}(2.5,2.5)(6.5,4)
 \rput(7,4){$x$}
 \rput(5,2){$x_{r}$}
 \psdots(5,2.5)
 \rput(5.1,7){$y_{r}$}
\psdots(4.5,7)
 \rput(1.5,6.5){$H_{a}$}
 \rput(2.5,2){$0$}
 \rput(3.3,3.8){$a$}
 \rput(4,0.5){Figure 1}
\end{pspicture}
\end{center}

Let us see that the distance is preserved under reflection. Let $x,y$ in $\mathbb{R}^{d}$, then
\begin{eqnarray*}
\langle x_{r},x_{r}\rangle  &=&\left\langle 2a+x-\frac{2\langle a,x\rangle }
{\langle a,a\rangle }a,2a+x-\frac{2\langle a,x\rangle }{\langle a,a\rangle }a\right\rangle  \\
&=&\langle x,x\rangle +4\left( \langle a,a\rangle -\langle a,x\rangle\right) , \\
\langle x_{r},y_{r}\rangle  &=&\left\langle 2a+x-\frac{2\langle a,x\rangle }
{\langle a,a\rangle }a,2a+y-\frac{2\langle a,y\rangle }{\langle a,a\rangle }
a\right\rangle  \\
&=&\langle x,y\rangle +2\left( 2\langle a,a\rangle -\langle a,x\rangle-\langle a,y\rangle \right) ,
\end{eqnarray*}
this implies that 
\begin{eqnarray}
|x_{r}-y_{r}|^{2} &=&\langle x_{r},x_{x}\rangle +\langle y_{r},y_{r}\rangle-2\langle x_{r},y_{r}\rangle   \notag \\
&=&|x-y|^{2}.  \label{disref}
\end{eqnarray}

The function $R_{a}:\mathbb{R}^{d}\rightarrow \mathbb{R}^{d}$, $R_{a}(x)=x_{r}$, is the reflection with respect to the affine hyperplane $H_{a}$. If $u\in \mathcal{S}(\mathbb{R}^{d})$ then the reflection $u_{r}$ of $u$ is $u\circ R_{a}$. We have $(-\Delta )^{s}(u\circ R_{a})=((-\Delta )^{s}u)\circ R_{a}$, that is to say the fractional Laplacian of the reflection is the reflection of the fractional Laplacian. Indeed, (\ref{disref}) yields
\begin{eqnarray*}
(-\Delta )^{s}u_{r}(x) &=&c_{d,s} \ \text{P.V.} \ \int_{\mathbb{R}^{d}}\frac{u(x_{r})-u(y_{r})}{|x-y|^{d+2s}}dy \\
&=&c_{d,s} \ \text{P.V.} \ \int_{\mathbb{R}^{d}}\frac{u(x_{r})-u(y_{r})}{|x_{r}-y_{r}|^{d+2s}}dy,
\end{eqnarray*}
now let us make the change of variable $z=y_{r}=R_{a}(y)$, then
\begin{eqnarray*}
\frac{\partial z_{i}}{\partial y_{i}} &=&1-\frac{2(a_{i})^{2}}{\langle
a,a\rangle }, \\
\frac{\partial z_{i}}{\partial y_{j}} &=&-\frac{2a_{i}a_{j}}{\langle
a,a\rangle },\text{ \ }j\neq i.
\end{eqnarray*}
Since the Jacobian is $|\det (R_{a}^{\prime })|=|-1|=1,$ then
\begin{eqnarray}
(-\Delta )^{s}u_{r}(x) &=&c_{d,s} \ \text{P.V.} \ \int_{\mathbb{R}^{d}}\frac{u(x_{r})-u(z)}{|x_{r}-z|^{d+2s}}dz  \notag \\
&=&(-\Delta )^{s}u(x_{r}).  \label{igualref}
\end{eqnarray}

Now let us suppose that (\ref{edif}) has a solution $u$ in $C^{2}(\overline{B_{1}})$. Let $x_{0}\in B_{1}\backslash \{0\}$ be an arbitrary fixed point. We will see that $u(x_{0})=u(0)$. We are going to consider the affine hyperplane $H_{\frac{1}{2}x_{0}}$ and the reflection $R_{\frac{1}{2}x_{0}}$ with respect to it. Let $v=u-u_{r}$, where $u_{r}=u\circ R_{\frac{1}{2}x_{0}}$. Applying Corollary \ref{CorApx} to the function $v$ in the ball $B_{r_{0}}(\frac{1}{2}x_{0})$, $r_{0}=1-\frac{1}{2}|x_{0}|$, see the Figure 2, we have that, for each $\varepsilon >0$, there is $g_{\varepsilon }\in C_{0}^{\infty }(U)$, $U=B_{4r_{0}}(\frac{1}{2}x_{0})\backslash \overline{B_{3r_{0}}(\frac{1}{2}x_{0})}$,
such that 
\begin{equation}
||v-K_{2s-d}g_{\varepsilon }||_{C^{2}\left( \overline{B_{r_{0}}(\frac{1}{2}x_{0})}\right) }<\varepsilon. \label{apxenbol}
\end{equation}
\begin{center}
\begin{pspicture}(0,0)(8,8)
 \rput(2.1,2){$B_{1}$}
 \psline{-}(2,7)(7,2)
 \psline{-}(4,4)(5,5)
 \rput(5.2,4.7){$x_{0}$}
 \psset{dotsize=2.5pt 0}
 \psdots(4,4)
 \psdots(5,5)
 \psdots(4.5,4.5)
 \pscircle(4,4){2}
 \psarc{-,linestyle=dashed}(5,5){2}{155}{-65} 
 \pscircle[linecolor=blue](4.5,4.5){1.3}
 \rput(1.5,6.5){$H_{\frac{1}{2}x_{0}}$}
 \rput(4,3.7){$0$}
 \rput(4.6,4){$\frac{1}{2}x_{0}$}
 \rput(4,0.5){Figure 2}
 \pscircle[linecolor=red](5,5){0.6}
\end{pspicture}
\end{center}

We are going to see that 
\begin{equation}
\lim_{\varepsilon \downarrow 0}(-\Delta )^{s}(v-K_{2s-d}g_{\varepsilon})(x_{0})=0.  \label{converg}
\end{equation}
For this purpose let us set 
\begin{equation*}
h_{\varepsilon }=\left\{ 
\begin{array}{ll}
v-K_{2s-d}g_{\varepsilon }, & \text{in }B_{r_{1}}(x_{0}), \\ 
0, & \text{in }\mathbb{R}^{d}\backslash B_{r_{1}}(x_{0}),\end{array}\right. 
\end{equation*}
where $r_{1}=1-|x_{0}|$. Then
\begin{align*}
(-\Delta )^{s}h_{\varepsilon }(x_{0})& =c_{d,s}\lim_{\zeta \downarrow0}\int_{B_{r_{1}}(x_{0})\backslash B_{\zeta }(x_{0})}\frac{h_{\varepsilon }(x_{0})-h_{\varepsilon }(y)}{|x_{0}-y|^{d+2s}}dy \\
& \ \ \ +c_{d,s}\int_{\mathbb{R}^{d}\backslash B_{r_{1}}(x_{0})}\frac{h_{\varepsilon }(x_{0})-h_{\varepsilon }(y)}{|x_{0}-y|^{d+2s}}dy \\
& =c_{d,s}(I_{1}+I_{2}).
\end{align*}
Odd symmetry of $\langle Dh_{\varepsilon }(x_{0}),y-x_{0}\rangle$ on $B_{r_{1}}(x_{0})\backslash B_{\zeta }(x_{0})$ gives us
\begin{equation*}
I_{1}=\lim_{\zeta \downarrow 0}\int_{B_{r_{1}}(x_{0})\backslash B_{\zeta }(x_{0})}\frac{h_{\varepsilon }(x_{0})-h_{\varepsilon
}(y)+\langle Dh_{\varepsilon }(x_{0}),y-x_{0}\rangle }{|x_{0}-y|^{d+2s}}dy.
\end{equation*}
From Taylor expansion of $h_{\varepsilon}$ at $x_{0}$ we get
\begin{align*}
\left\vert I_{1}\right\vert & \leq \lim_{\zeta \downarrow 0}\int_{B_{r_{1}}(x_{0})\backslash B_{\zeta }(x_{0})}\frac{||D^{2}h_{\varepsilon }||_{L^{\infty }(B_{r_{1}}(x_{0}))} \ |y-x_{0}|^{2}}{|x_{0}-y|^{d+2s}}dy \\
& =||D^{2}h_{\varepsilon }||_{L^{\infty }(B_{r_{1}}(x_{0}))}\frac{r_{1}^{2(1-s)}}{2(1-s)}\\
&\leq \frac{r_{1}^{2(1-s)}}{2(1-s)}||v-K_{2s-d}g_{\varepsilon }||_{C^{2}(\overline{B_{r_{0}}(\frac{1}{2}x_{0})})}.
\end{align*}
Furthermore, using that $h_{\varepsilon}=0$ in $\mathbb{R}^{d}\backslash B_{r_{1}}(x_{0})$ we have
\begin{eqnarray*}
|I_{2}| &\leq &\int_{\mathbb{R}^{d}\backslash B_{r_{1}}(x_{0})}\frac{||h_{\varepsilon }||_{L^{\infty }(B_{r_{1}}(x_{0}))}}{|x_{0}-y|^{d+2s}}dy \\
&=&\frac{r_{1}^{-2s}}{2s}||v-K_{2s-d}g_{\varepsilon }||_{C^{2}(\overline{B_{r_{0}}(\frac{1}{2}x_{0})})}.
\end{eqnarray*}
In this way convergence (\ref{converg}) is followed by (\ref{apxenbol}).

On the other hand, using (\ref{defLa}) and (\ref{tfonor}) we get
\begin{eqnarray*}
(-\Delta )^{s}(K_{2s-d}\ast g_{\varepsilon })(x) &=&\mathcal{F}^{-1}((2\pi|\xi |)^{2s}\mathcal{F}(K_{2s-d}\ast g_{\varepsilon })(\xi ))(x) \\
&=&\mathcal{F}^{-1}((2\pi |\xi |)^{2s}\mathcal{F}(K_{2s-d})(\xi )\mathcal{F}(g_{\varepsilon })(\xi ))(x) \\
&=&\frac{(2\pi )^{d}}{c_{d,s}}g_{\varepsilon }(x).
\end{eqnarray*}
Using the lineality of the fractional Laplacian operator, (\ref{converg}) and (\ref{igualref}) we obtain
\begin{eqnarray*}
0 &=&\lim_{\varepsilon \rightarrow 0}(-\Delta )^{s}(v-K_{2s-d}g_{\varepsilon})(x_{0}) \\
&=&(-\Delta )^{s}v(x_{0})-\lim_{\varepsilon \rightarrow 0}(-\Delta)^{s}(K_{2s-d}g_{\varepsilon })(x_{0}) \\
&=&(-\Delta )^{s}u(x_{0})-(-\Delta )^{s}u_{r}(x_{0})-\frac{(2\pi )^{d}}{c_{d,s}}g_{\varepsilon }(x_{0})\\
&=&(-\Delta )^{s}u(x_{0})-(-\Delta )^{s}u((x_{0})_{r}) \\
&=&f(u(x_{0}))-f(u((x_{0})_{r})),
\end{eqnarray*}
we have used $x_{0}\in B_{r_{0}}(\frac{1}{2}x_{0})$ and that the support of $g_{\varepsilon }$ is contained in $B_{4r_{0}}(\frac{1}{2}x_{0})\backslash \overline{B_{3r_{0}}(\frac{1}{2}x_{0})}$. From (\ref{defref}) we see  $(x_{0})_{r}=0$, then $f(u(x_{0}))=f(u(0))$. 

The continuity of $u$ implies that $u(B_{1})$ is a connected set in $\mathbb{R}$, therefore it is an interval. If $(u(B_{1}))^{\circ }=\emptyset $, then $u(x)=c$, for each $x\in B_{1}$ and some constant $c\in \mathbb{R}$. If $c=0$, then $0=(-\Delta)^{s}u(0)=f(u(0))=f(0)$, then it must be $c\neq 0$. On the other hand
\begin{eqnarray*}
f(c)&=&f(-\Delta )^{s}u(0) \\
&=&c_{d,s}\lim_{\zeta \downarrow 0}\int_{B_{1}\backslash B_{\zeta }}\frac{u(0)-u(y)}{|y|^{d+2s}}dy
+c_{d,s}\int_{\mathbb{R}^{d}\backslash B_{1}}\frac{u(0)-u(y)}{|y|^{d+2s}}dy \\
&=&c_{d,s} \ c \int_{\mathbb{R}^{d}\backslash B_{1}}\frac{1}{|y|^{d+2s}}dy=\infty .
\end{eqnarray*}
Thus $(u(B_{1}))^{\circ }\neq \emptyset $, but this is also impossible because $f$ would be constant in the open interval $(u(B_{1}))^{\circ }$ and 
$f$ is in $\mathcal{J}$. \hfill
\end{proof}

\section*{Acknowledgment} The author was partially supported by the grant PIM20-1 of Universidad Aut\'{o}noma de Aguascalientes.


\begin{thebibliography}{00}

\bibitem{BLMW} M. Birkner, J.A. López-Mimbela, A. Wakolbinger. {\it Comparison results and steady states for the Fujita equation with fractional Laplacian}. Ann. Inst. Henri Poincaré, Anal. Non Linéaire, 22 (2005),  83-97.

\bibitem{CRS} L.A. Caffarelli, J.M. Roquejoffre, Y. Sire. {\it Variational problems for free boundaries for the fractional Laplacian}. J. Eur. Math. Soc., 12(5) (2010), 1151–1179.

\bibitem{CSS} L. Caffarelli, S. Salsa, L. Silverstre. {\it Regularity estimates for the solution and the free boundary of the obstacle problem for the fractional Laplacian}. Invent. Math. 171 (2008), 425-461.

\bibitem{CFY} W. Chen, Y. Fang, R. Yang. {\it Semilinear equations involving the fractional Laplacian on domains} (2013) arXiv:1309.7499v1.
\bibitem{CLL} W. Chen, C. Li, Y. Li. {\it A direct method of moving planes for the fractional Laplacian}. Advances in Mathematics Vol. 308 (2017), 404-437.

\bibitem{DSV} S. Dipierro, O. Savin, E. Valdinoci. {\it All functions are locally $s$-harmonic up to a small error}. J. Eur. Math. Soc. 19(4) (2017), 957-966. 

\bibitem{EK} S.N. Ethier, T.G. Kurtz. Markov processes: Characterization and convergence. John Wiley $\&$ Sons, 1986.

\bibitem{Fall} M.M. Fall. {\it Semilinear elliptic equations for the fractional Laplacian with Hardy potential}. Nonlinear Analysis (2018) https://doi.org/10.1016/j.na.2018.07.008.

\bibitem{FallWeth} M.M. Fall, T. Weth. {\it Nonexistence results for a class of fractional elliptic boundary value problems}. J. Funct. Anal. 263(8) (2012), 2205-2227.

\bibitem{Get} R.K. Getoor. {\it First passage times for symmetric stable processes in space}. Transactions of the American Mathematical Society Vol. 101, No. 1 (1961),  75-90.

\bibitem{IM} C. Imbert, R. Monneau. {\it Homogenization of first-order equations with $(u/\varepsilon)$-periodic Hamiltonians. I. Local equations}. Arch. Ration. Mech. Anal., 187(1) (2008), 49–89.

\bibitem{Kry} N.V. Krylov. {\it On the paper "All functions are locally s-harmonic up to a small error" by Dipierro, Savin, and Valdinoci} (2018) arXiv:1810.07648v1.

\bibitem{MW} J.A. Jr Mann, W.A. Woyczynski. {\it Growing fractal interfaces in the presence of self-similar hopping surface diffusion}. Physica A (2001), 159–183.

\bibitem{PR} F. Pacella, M. Ramaswamy. {\it Symmetry of solutions of elliptic equations via maximum principles}. Handbook of Differential Equations (M. Chipot, ed.), Elsevier (2012), 269-312.

\bibitem{RosSerra} X. Ros-Oton, J. Serra. {\it The Pohozaev identity for the fractional Laplacian}. Archive for Rational Mechanics and Analysis, 213(2) (2014), 587–628.

\bibitem{SeVa1} R. Servadei, E. Valdinoci. {\it The Brezis-Nirenberg result for the fractional Laplacian}. Transactions of the American Mathematical Society Vol. 367, No. 1 (2015),  67-102.

\bibitem{SeVa2} R. Servadei, E. Valdinoci. {\it Variational methods for non-local operators of elliptic type}. Discrete and Continuous Dynamical Systems-A, 33(5) (2013), 2105-2137.

\bibitem{SZF} M.F. Shlesinger, G.M. Zaslavsky, U. Frisch (eds). Lévy Fligths and Related Topics in Physics, Lecture Notes in Physics, Vol. 450. Springer-Verlag: Berlin, 1995.

\bibitem{Toland} J. F. Toland. {\it The Peierls-Nabarro and Benjamin-Ono equations}. J. Funct. Anal., 145(1) (1997), 136–150. 
\end{thebibliography}
\end{document}